\documentclass[12pt]{amsart}
\usepackage{fullpage}
\usepackage[T1]{fontenc}
\usepackage{amsmath,amssymb,amsthm,mathtools,booktabs,array}
\usepackage{needspace}
\usepackage[colorlinks=true,linkcolor=red,urlcolor=blue,citecolor=blue]{hyperref}
\hypersetup{pdftitle={Finite Symplectic Groups on Hyperkahler Manifolds of K3[n]-Type Act Trivially on Discriminant Groups}}

\makeatletter
\def\thm@space@setup{%
  \thm@preskip=2ex \thm@postskip=2ex
}
\makeatother
\numberwithin{equation}{section}
\theoremstyle{plain}
\newtheorem{theorem}{Theorem~}[section]
\newtheorem{proposition}[theorem]{Proposition~}
\newtheorem{lemma}[theorem]{Lemma~}

\theoremstyle{definition}
\newtheorem{definition}{Definition~}[section]
\theoremstyle{remark}
\newtheorem{remark}[theorem]{Remark~}
\newcommand{\Z}{\mathbb Z}\newcommand{\Q}{\mathbb Q}\newcommand{\R}{\mathbb R}
\newcommand{\F}{\mathbb F}
\newcommand{\rk}{\operatorname{rk}}\newcommand{\divi}{\operatorname{div}}
\newcommand{\ord}{\operatorname{ord}}
\newcommand{\im}{\operatorname{im}}
\newcommand{\M}{{\mathrm{II}_{4,20}}}\newcommand{\id}{\mathrm{id}}

\title{Finite Symplectic Automorphisms of Hyperk\"ahler Manifolds of $K3^{[n]}$-Type Act Trivially on Discriminant Groups}
\author{Jie Fu, Shihao Wang, Zhiwei Zheng}
\date{}
\newcommand{\Addresses}{{%
    \bigskip
    \footnotesize

    J.~Fu, \textsc{Qiuzhen college, Tsinghua University, Beijing, China}\par\nopagebreak
    \textit{Email address}: \texttt{fu-j21@mails.tsinghua.edu.cn}

    \medskip

    S.~Wang, \textsc{Qiuzhen college, Tsinghua University, Beijing, China}\par\nopagebreak
    \textit{Email address}: \texttt{wangshihao25@mails.tsinghua.edu.cn}

    \medskip

    Z.~Zheng, \textsc{Yau Mathematical Sciences Center, Tsinghua University, Beijing, China}\par\nopagebreak
    \textit{Email address}: \texttt{zhengzhiwei@mail.tsinghua.edu.cn}
}}

\begin{document}
\begin{abstract}
We give an alternative proof of Mongardi's result that every finite group of symplectic automorphisms of a hyperk\"ahler manifold of $K3^{[n]}$-type acts trivially on the discriminant groups of its second cohomology lattice and its coinvariant lattice. The proof combines sphere-packing bounds, Vinberg's reflection chambers and structure results on even $2$-elementary lattices.
\end{abstract}

\maketitle
\setcounter{tocdepth}{2}
\tableofcontents

\section{Introduction}\label{sec:intro}
The study of symplectic automorphisms of hyperk\"ahler manifolds has been an active area of research in recent years. Let $X$ be a hyperk\"ahler manifold, and let $G\subset\operatorname{Aut}^s(X)$ be a finite group of symplectic automorphisms. Let
\[
 S=S_G(H^2(X,\Z))=(H^2(X,\Z)^G)^\perp
\]
be the coinvariant lattice for the induced action of $G$ on $H^2(X,\Z)$. The lattice $S$ is always negative definite. If, moreover, the induced action of $G$ on the discriminant group $A_{H^2(X,\Z)}$ is trivial, then its action on $A_S$ is also trivial. For $X$ of $K3^{[n]}$-type, Mongardi \cite[Theorem~25]{MR3473636} shows that $S(-1)$, obtained by reversing the sign of the pairing on $S$, embeds primitively into the positive-definite Leech lattice $\mathbb{L}$ as a coinvariant sublattice, so it is one of the $290$ lattices classified by H\"ohn--Mason \cite{MR3438323}. He also conjectures that $\rk(S)+\ell(A_S)<24$ \cite[Conjecture~32]{MR3473636}, where $\ell(A_S)$ is the minimal number of generators of $A_S$. We prove this conjecture in a another paper \cite{FuWangZheng}.

For a hyperk\"ahler manifold $X$ of $K3^{[n]}$-type with $n\ge2$, write $L_n=H^2(X,\Z)$. Then $A_{L_n}\cong \Z/(2n-2)\Z$. Markman \cite[Lemma~9.2]{MR2964480} proved that every monodromy operator, and hence every automorphism of $X$, acts on this group as $\pm\id$. Mongardi \cite[Lemma 24(4)]{MR3473636} stated that any finite symplectic automorphism induces trivial action on the discriminant group of $H^2(X,\Z)$. The argument relies on \cite[Lemma 23]{MR3473636}, but we have been unable to justify the assertion in its proof that $S_g(L_n)$ is negative definite (the subtlety is that $g$ may have infinite order). This paper aims to provide an alternative proof of Mongardi's result:

\begin{theorem}\label{thm:main}
Let $X$ be a hyperk\"ahler manifold of $K3^{[n]}$-type, and let $G\subset\operatorname{Aut}^s(X)$ be a finite subgroup. Then the induced action of $G$ on $A_{H^2(X,\Z)}$ is trivial. This implies that the induced action on $A_{S_G(H^2(X,\Z))}$ is also trivial.
\end{theorem}

\begin{definition}\label{def:short}
Let $N$ be a hyperbolic lattice. We say that $N$ satisfies property~\eqref{eq:W} if, for every primitive vector $v\in N$ with even square $v^2\ge4$, there exists $a\in N$ such that
\begin{equation}\tag{W}\label{eq:W}
 a\notin\Q v,\qquad
 a^2\ge-2,\qquad |(a,v)|\le \frac{v^2}{2}.
\end{equation}
Such an $a$ is called a \emph{short vector} for $(N,v)$.
\end{definition}

We write $\mathrm{II}_{p,q}$ for the even unimodular lattice of signature $(p,q)$, where $p,q\ge1$ and $p\equiv q\pmod{8}$. For a single isometry $g$ of a lattice $L$, write $S_g(L)=(L^g)^\perp_L$.

\begin{theorem}\label{thm:lattice}
Let $F\in O(\mathrm{II}_{4,20})$ have finite order and fix a positive-definite three-dimensional subspace pointwise. If $v$ is primitive, $v^2=2(n-1)\ge4$, and $F(v)=-v$, then $v$ has a short vector in $S_F(\M)$.
\end{theorem}
\begin{proof}
Write $\ord(F)=2^kq$ with $q$ odd. Since $F(v)=-v$, we have $k\ge1$, $F^q(v)=-v$, and $S_{F^q}(\M)\subset S_F(\M)$; replace $F$ by $F^q$. If $k=1$, Proposition~\ref{prop:two-elementary} gives a solution of \eqref{eq:W}. Otherwise, suppose there is no short vector. Lemma~\ref{lem:E8} gives $\ord(F)\in\{4,8\}$, contradicted by Propositions~\ref{prop:orderfour} and \ref{prop:remaining}, respectively.
\end{proof}

\begin{remark}
For $n\ge3$, Theorem~\ref{thm:lattice} is equivalent to
\cite[Lemma~23]{MR3473636}.
This follows from Definition~\ref{def:numerical-wall}.
\end{remark}

A lattice $L$ is $2$-elementary if $A_L$ has exponent at most two. For involutions of $\M$, the coinvariant lattice is even and $2$-elementary. The Cohn--Elkies bounds \cite{MR1973059} establish property~\eqref{eq:W} outside four exceptional pairs of invariants, which are treated using Nikulin's classification \cite{MR525944} and Vinberg's reflection chambers \cite{MR295193}. Our use of sphere-packing bounds is inspired by Giovenzana, Grossi, Onorati and Veniani's proof of symplectic rigidity for hyperk\"ahler manifolds of OG10-type \cite{MR4753270}.

For higher two-power orders, we first use the classification of Leech involutions \cite{MR2848589}. Unless this already produces a short vector, unimodular gluing gives an even unimodular lattice $K\simeq\mathrm{II}_{4,12}$ on which $F$ induces an isometry $g$ with $\ord(g)=\ord(F)/2$. We construct suitable vectors in $K$ using the structure of $2$-elementary lattices and lift them to short vectors in $\M$.

\medskip
\noindent\textbf{Acknowledgements.} We thank Giovanni Mongardi for helpful correspondence when this work was nearing completion. We know from him that he had independent, unpublished arguments repairing the gap discussed above. He also emphasized the relevance of \cite{MR4753270}.

\medskip
\noindent \textbf{AI disclosure.} LLM-based AI tools were used in the course of this research to explore proof strategies, assist
in the development and refinement of mathematical arguments, and improve the exposition. The authors take full responsibility for all results and arguments presented in this paper.

\section{Order Two Case}\label{sec:involution}
\subsection{Sphere Packing Method}\label{sec:sphere-packing}
We first recall the Cohn--Elkies sphere-packing bound. For $1\le d\le22$, let $\delta_d^{\mathrm{CE}}$ denote the center-density bound for dimension $d$ in \cite[Table~3]{MR1973059}, increased by $10^{-5}$ to ensure a valid upper bound. Every packing of unit balls in $\R^d$ has center density at most $\delta_d^{\mathrm{CE}}$. In particular, if $P$ is a Euclidean lattice of rank $d$ with $\min(P)\ge4$, where $\min(P):=\min\{(x,x):0\ne x\in P\}$, then
\[
 \frac1{\sqrt{\det P}}\le\delta_d^{\mathrm{CE}}.
\]

Inspired by the sphere-packing argument in \cite[\S\S~2.4 and 3.2]{MR4753270}, the following lemma uses the Cohn--Elkies bounds to study $2$-elementary hyperbolic lattices.

For a finite abelian group $A$, let $\ell(A)$ denote its minimal number of generators.

\begin{lemma}\label{lem:volume}
Let $N$ be an even hyperbolic $2$-elementary lattice with $\rk(N)+\ell(A_N)\le24$. If
\[
(\rk(N),\ell(A_N))\notin
\{(11,11),(12,12),(13,11),(14,10)\},
\]
then $N$ has property~\eqref{eq:W}.
\end{lemma}
\begin{proof}
Put $r=\rk(N)$ and $a=\ell(A_N)$. If $r=1$, then $N\simeq\langle2\rangle$, and property~\eqref{eq:W} holds vacuously. The hypothesis gives $r\le24$. If $r=24$, then $a=0$ and $N$ is even unimodular, which is impossible since its signature difference $1-23=-22$ is not divisible by $8$. Thus $2\le r\le23$.

Fix a primitive $v\in N$ with $v^2=2(n-1)\ge4$, and suppose it has no short vector. Let $P$ be the Euclidean lattice whose underlying group is the orthogonal projection of $N$ onto $v^\perp$, equipped with the bilinear form on $N$ rescaled by $-1$. Since $v$ is primitive, it extends to a basis $v,e_2,\ldots,e_{\rk(N)}$ of $N$. The projections of the $e_i$ form a basis of the Euclidean lattice $P$, with Gram matrix
\[
\left(\frac{(e_i,v)(e_j,v)}{v^2}-(e_i,e_j)\right)_{i,j\ge2}.
\]
Replacing each $e_i$ by $e_i-\frac{(e_i,v)}{v^2}v$ makes the Gram matrix block diagonal without changing its determinant. Thus $|\det N|=v^2\det P$, and
\[
\det P=\frac{|\det N|}{v^2}
      =\frac{|A_N|}{v^2}
      =\frac{2^a}{2(n-1)}.
\]
Let $0\ne y\in P$ and choose a lift $x\in N$. Adding an integer multiple of $v$ leaves its projection unchanged and changes $(x,v)$ by multiples of $v^2=2(n-1)$, so we may arrange $|(x,v)|\le n-1$. Since $y\ne0$, we have $x\notin\Q v$. The assumption that there is no short vector gives $x^2<-2$, hence $x^2\le-4$. Thus
\[
\|y\|_P^2=-x^2+\frac{(x,v)^2}{v^2}\ge4,
\]
so $\min(P)\ge4$. For distinct $x,y\in P$, we have $\|x-y\|\geq2$. Thus balls of radius $1$ centered at the lattice points have disjoint interiors. The lattice $P$ has rank $r-1$ and covolume $\sqrt{\det P}$, so this packing has center density $1/\sqrt{\det P}$. The Cohn--Elkies bound in \cite[Table~3]{MR1973059} therefore gives
\[
\sqrt{\frac{2(n-1)}{2^a}}
=\frac1{\sqrt{\det P}}
\le\delta_{r-1}^{\mathrm{CE}}.
\]
Squaring both sides yields $2(n-1)/2^a\le(\delta_{r-1}^{\mathrm{CE}})^2$, and multiplying by $2^{a-1}$ gives
\[
n-1\le 2^{a-1}\bigl(\delta_{r-1}^{\mathrm{CE}}\bigr)^2.
\]
Since $A_N\simeq(\Z/2\Z)^a$, the Smith normal form of a Gram matrix of $N$ has $r-a$ entries equal to $1$ and $a$ entries equal to $2$. Its rank modulo $2$ is therefore $r-a$. As $N$ is even, the Gram matrix modulo $2$ is alternating and has even rank, so $a\equiv r\pmod{2}$. Outside the $4$ excluded pairs, $a\le\min\{r,24-r\}-2$ for $11\le r\le14$, while $a\le\min\{r,24-r\}$ otherwise. The bounds in \cite[Table~3]{MR1973059} give the following estimates for the right-hand side above, omitting the $4$ excluded pairs:
\[
\begin{array}{cccc}
\toprule
\text{range of }r & \text{maximizing }(r,a) & \delta_{r-1}^{\mathrm{CE}} & \text{bound on }n-1\\
\midrule
2\le r\le10 & (10,10) & 0.05901 & <1.79\\
11\le r\le14 & (12,10) & 0.05933 & <1.81\\
15\le r\le21 & (15,9) & 0.07751 & <1.54\\
22 & (22,2) & 0.37390 & <0.28\\
23 & (23,1) & 0.51232 & <0.27\\
\bottomrule
\end{array}
\]
This is a contradiction to $n-1\ge2$.
\end{proof}

\begin{remark}\label{rem:projection}
In \cite[\S~3.2]{MR4753270}, sphere-packing bounds are applied directly to the coinvariant lattice after reversing its sign. Here we adapt and refine this approach by using the projection lattice $P$.

Following the notation of the above proof, set $S=N\cap v^\perp$ and $t=\divi_N(v)$. Write $L(m)$ for $L$ with its bilinear form multiplied by a nonzero integer $m$. The projection lattice $P$ contains $S(-1)$ with index $v^2/t$, and
\[
 |\det S|=\frac{v^2\,2^a}{t^2},\qquad
 \det P=\frac{2^a}{v^2}.
\]
If $v$ has no short vector, then
$\min P\geq4$ and $\min S(-1)\geq4$. A sphere-packing bound on $S(-1)$ gives only a lower bound on $v^2$, while the same bound on $P$ gives the upper bound needed for the contradiction.
\end{remark}

\begin{remark}
Blichfeldt's bound \cite{MR1512555} gives instead
\[
n-1\le\frac{2^a\Gamma((\rk(N)+3)/2)^2}{2(2\pi)^{\rk(N)-1}}.
\]
This already excludes the ranges $2\le\rk(N)\le9$ and $16\le\rk(N)\le23$, but leaves the intermediate ranks unresolved.

The Rogers--Leech bounds \cite{MR102052}, \cite{MR209983} also prove Lemma~\ref{lem:volume}. Using the values in \cite[Table~1]{MR4753270}, increased by $10^{-4}$ to ensure a valid upper bound, the same projection estimate gives $n-1<1.94$ outside the $4$ exceptional pairs.
\end{remark}

\subsection{Vinberg's Inequality}\label{sec:vinberg}
We now recall Vinberg's reflection chambers for $I_{1,d}$, the odd unimodular lattice of signature $(1,d)$. Vinberg proves that these chambers have finite volume for $2\le d\le17$ \cite[\S~4.1]{MR295193}. For $11\le d\le14$, let $\mathcal C_d\subset\R^{1,d}$ be the cone defined by the linear inequalities
\[
 a\ge0,\qquad c_1\ge\cdots\ge c_d\ge0,\qquad
 a\ge c_1+c_2+c_3,\qquad 3a\ge\sum_{i=1}^{11}c_i,
\]
and, when $d=14$, also $4a\ge2c_1+\sum_{i=2}^{14}c_i$.
Vinberg uses the quadratic form $-a^2+\sum_{i=1}^d c_i^2$, whereas we use its negative. In our sign convention, the inequalities above are the wall inequalities of these chambers \cite[\S~4.1, (56) and Table~4]{MR295193}. Since these chambers have finite volume, the cones defined by their wall inequalities lie in the closure of the positive cone \cite[\S~1.1, (13)]{MR295193}. Thus
\begin{equation}\label{eq:vinberg-cone}
 (a;c_1,\ldots,c_d)\in\mathcal C_d
 \quad\Longrightarrow\quad a^2-\sum_{i=1}^d c_i^2\ge0.
\end{equation}

\begin{lemma}\label{lem:fourpairs}
Let $N$ be an even hyperbolic $2$-elementary lattice with
$(\rk N,\ell(A_N))\in\{(11,11),(12,12),(13,11),(14,10)\}$.
Then $N$ has property~\eqref{eq:W}.
\end{lemma}
\begin{proof}
The models \cite[Theorem~3.6.2]{MR525944} for the first three pairs are $I_{1,10}(2)$, $I_{1,11}(2)$, and $U\oplus A_1(-1)^{11}$; the last pair gives $U\oplus D_4(-1)\oplus E_8(-2)$ or $U\oplus D_4(-1)\oplus A_1(-1)^8$. Each model contains a full-rank sublattice $\mathcal F=I_{1,s}(2)$ with $s=\rk N-1\in\{10,11,12,13\}$: this follows from $I_{1,1}(2)\subset U$, $A_1(-1)^4\subset D_4(-1)$, and $\langle2\rangle\oplus E_8(-2)\simeq I_{1,8}(2)$. Since $\mathcal F\subset N\subset\mathcal F^\vee$, a primitive vector of square $2(n-1)\ge4$ has coordinates
\[
 v=(a;b_1,\ldots,b_s),\qquad n-1=a^2-\sum_{i=1}^s b_i^2\in\Z_{\ge2},
 \qquad a,b_i\in\tfrac12\Z.
\]
Replace $v$ by $-v$ if necessary, then change signs and permute the $b_i$ to arrange $a>0$ and $b_1\ge\cdots\ge b_s\ge0$. All sign changes and coordinate permutations are applied simultaneously to $N$ and $v$. Set
\[
 u_1=(1;1,1,1,0^{s-3}),\qquad
 u_2=\begin{cases}(3;1^{10}),&s=10,\\(3;1^{11},0^{s-11}),&s\ge11.\end{cases}
\]
If $(v,u_j)<0$ for $j=1$ or $2$, apply the reflection
\[
 s_{u_j}(x)=x-\frac{2(x,u_j)}{u_j^2}u_j.
\]
Since $u_j^2$ is $-2$ or $-4$ and $(\mathcal F,u_j)\subset2\Z$, these reflections preserve $\mathcal F$ and $\mathcal F^\vee$. They also preserve the component of the positive cone containing $v$, so the first coordinate of $s_{u_j}(v)$ remains positive. It is strictly smaller than $a$: the coefficient $2(v,u_j)/u_j^2$ and the first coordinate of $u_j$ are both positive.

Apply each reflection to both $N$ and $v$, and then apply sign changes and permutations again so that
$b_1\geq\cdots\geq b_s\geq0$. These operations preserve $\mathcal F\subset N\subset\mathcal F^\vee$. The process terminates because $a$ remains a positive half-integer and strictly decreases at each reflection \cite{MR690711}, \cite{MR295193}. At the end of this process, $(v,u_1)$ and $(v,u_2)$ are both nonnegative, which means
\begin{equation}\label{eq:frame-reduced}
 a\ge b_1+b_2+b_3,\qquad
 3a\ge\sum_{i=1}^{\min(s,11)}b_i.
\end{equation}
All transformations are isometries, so a short vector for the resulting pair pulls back to one for the original $(N,v)$.

Suppose that $v$ has no short vector. Consider the root $r=(0;0^{s-1},1)\in\mathcal F\subset N$, then $(v,r)=-2b_s$. Then $2b_s\ge n$, hence $b_s\ge n/2$. Then adding $n/2$ after the last coordinate preserves the coordinate ordering:
\[
 \widehat v=(a;b_1,\ldots,b_s,\frac{n}{2})\in\R^{1,s+1}.
\]
For $s=11,12$, the ordering and \eqref{eq:frame-reduced} give $\widehat v\in\mathcal C_{s+1}$. For $s=10$, the vector $u_2$ has square $-2$ and $(v,u_2)=2(3a-\sum_{i=1}^{10}b_i)\ge0$. Since there is no short vector, we have
\[
 3a-\sum_{i=1}^{10}b_i\ge\frac{n}{2},
\]
which gives the remaining wall inequality for
$\widehat v\in\mathcal C_{11}$.

For $s=13$, take $u=(4;2,1^{12})\in\mathcal F\subset N$. This vector is nonzero and has square $0$, so $u\notin\Q v$ because $v^2>0$. It lies in the closure of the positive cone containing $v$, so $(v,u)>0$. Since $u$ is not short and $(v,u)$ is a positive integer,
\[
 4a-2b_1-\sum_{i=2}^{13}b_i=\frac{(v,u)}{2}\ge\frac{n}{2},
\]
which gives the remaining wall inequality for $\widehat v\in\mathcal C_{14}$. Thus $\widehat v\in\mathcal C_{s+1}$ in every case, and \eqref{eq:vinberg-cone} implies a contradiction to
\[
 a^2-\sum_{i=1}^s b_i^2-\left(\frac{n}{2}\right)^2
 =n-1-\left(\frac{n}{2}\right)^2
 =-\frac{(n-2)^2}{4}<0.
\]
\end{proof}

\begin{proposition}\label{prop:two-elementary}
Let $N$ be an even hyperbolic $2$-elementary lattice with $\rk(N)+\ell(A_N)\le24$. Then $N$ has property~\eqref{eq:W}.

In particular, if $F\in O(\M)$ is an involution fixing a positive-definite three-space pointwise and $S_F(\M)$ is hyperbolic, then $S_F(\M)$ has property~\eqref{eq:W}.
\end{proposition}
\begin{proof}
The first assertion follows from Lemmas~\ref{lem:volume} and \ref{lem:fourpairs}. For the involution case, unimodular gluing makes $N=S_F(\M)$ even and $2$-elementary, with $\rk(N)+\ell(A_N)\le24$, so the first assertion applies.
\end{proof}

\section{Higher Order Cases}\label{sec:higher}
\subsection{The Unimodular Reduction}\label{sec:reduction}
We first recall the Leech embedding lemma.

\begin{lemma}\label{lem:leech-embedding}
Let $S$ be a positive-definite even lattice with no vectors of square $2$, and let $G\subset O(S)$ be a finite subgroup acting trivially on $A_S$ and satisfying $S^G=0$. If $\rk S+\ell(A_S)\le24$, then $S$ admits a primitive embedding into the positive-definite Leech lattice $\mathbb{L}$. Moreover, the action of $G$ extends to $\mathbb{L}$ and fixes $S^\perp_{\mathbb{L}}$ pointwise.
\end{lemma}
This lemma was originally due to \cite[Appendix~B.2]{MR2955931}. For the primitivity issue in the original argument, see \cite[Example~4.12 and Corollary~4.19]{MR4931809} and \cite[Lemma~1.1]{Zheng}.

Throughout this section, let $F$ and $v$ satisfy the hypotheses of Theorem~\ref{thm:lattice}.

\begin{lemma}\label{lem:E8}
Assume $\ord(F)=2^k$ with $k\ge2$, let $N=S_F(\M)$, and suppose $(N,v)$ has no short vector. Set $H_0=F^{2^{k-1}}$. Then $S_{H_0}(\mathrm{II}_{4,20})\simeq E_8(-2)$ and $\ord(F)\in\{4,8\}$.
\end{lemma}
\begin{proof}
Write $C=S_{H_0}(\M)$ and $R=C^\perp_{\M}$.
Since $H_0(v)=v$, the fixed positive three-space together with $v$ gives a positive four-space fixed by $H_0$. Thus $C$ is negative definite, has rank at most $20$, and lies in $N\cap v^\perp$. This lattice has no vector of square $-2$, since such a vector would satisfy \eqref{eq:W}. Moreover,
\[
 \ell(A_C)=\ell(A_R)\le\rk R=24-\rk C.
\]
Since $H_0$ acts trivially on $A_C$, Lemma~\ref{lem:leech-embedding} applies to $(C(-1),\langle H_0\rangle)$. The classification of Leech involutions \cite[Appendix F, Theorem F.4]{MR2848589} (see also \cite[Table~1]{MR3438323}) gives
\[
 C(-1)\simeq E_8(2),\quad BW_{16},\quad D_{12}^+(2),\quad\mathbb{L}.
\]
Leech lattice $\mathbb{L}$ itself is excluded by rank. For $BW_16$ and $D_{12}^{+}(2)$, $\rk R$ is respectively $8$ and $12$, and their discriminant groups satisfy
\[
 A_R\simeq A_C\simeq(\Z/2\Z)^{\rk R}.
\]
Since $A_R$ is $2$-elementary, $R^\vee\subset\frac12R$. Both $R^\vee$ and $\frac12R$ contain $R$ with index $2^{\rk R}$, so $R^\vee=\frac12R$. Hence $v/2\in R^\vee$ because $v\in R$. Since $v$ is primitive in $\M$, we have $v/2\notin R$, so $v/2+R$ is a nonzero discriminant class.

For $L=BW_{16}$ or $D_{12}^{+}(2)$, every $y\in L_\R$ admits $x\in L$ such that $(y-x)^2\leq3$. For $BW_{16}$, rescale the bilinear form so that $\min BW_{16}=4$; see \cite[proof of Theorem~3.1]{MR1760630}. The covering radius of $D_{12}^+$ is $\sqrt{12/8}$
\cite[\S~III.B]{MR1664114}, hence the covering radius of $D_{12}^+(2)$ is $\sqrt3$. Hence every discriminant class of both lattices has a representative of square at most $3$.

Choose $c\in C^\vee$ with $c^2\geq-3$ such that $c+C$ corresponds to $v/2+R$ under the gluing anti-isometry. Integral gluing gives $a=v/2+c\in N$ with $(a,v)=n-1$. Moreover,
\[
 a^2\ge \frac{n-1}{2}-3\ge-2.
\]
The class $v/2+R$ is nonzero, so $a\notin\Q v$. Thus $a$ is a short vector for $(N,v)$, a contradiction to the standing assumption. Consequently $C\simeq E_8(-2)$.

Rescaling $C$ gives an isometry $u\in O(E_8)$ with $u^{2^{k-1}}=-1$. Its eigenvalues are therefore primitive $2^k$-th roots. Thus its characteristic polynomial is a power of the cyclotomic polynomial $\Phi_{2^k}(t)$, whose degree is $2^{k-1}$. Since $\rk E_8=8$, we have $2^{k-1}\mid8$. The only possibilities are $k=2,3,4$.

Suppose $k=4$. Then the characteristic polynomial is $t^8+1$. Put $T=1-u$; we obtain
\[
 [E_8:TE_8]=\det T=2.
\]
Choose $x\in E_8\setminus TE_8$. Since $E_8/TE_8$ has order $2$, there is $z\in E_8$ with $Tz=2x$. If $z=2y$ for some $y\in E_8$, then $Tz=2x$ gives $x=Ty\in TE_8$, a contradiction. Thus $z\notin2E_8$.

For every $y\in E_8$, the isometry $u$ gives
\[
 (Ty,z)=(y,(1-u^{-1})z)=-2(y,u^{-1}x)\in2\Z.
\]
Thus $TE_8$ is contained in the kernel of the homomorphism $E_8\to\Z/2\Z$, $y\mapsto(y,z)\pmod{2}$. This homomorphism is nonzero: otherwise $z/2\in E_8^\vee=E_8$, a contradiction to $z\notin2E_8$. Its kernel therefore has index $2$. Since $TE_8$ also has index two, it equals this kernel. Our choice $x\notin TE_8$ therefore forces $(x,z)$ to be odd.

On the other hand, $uz=z-2x$ and preservation of the square give
\[
 z^2=(uz)^2=(z-2x)^2=z^2-4(x,z)+4x^2.
\]
Hence $(x,z)=x^2$, which is even because $E_8$ is even. This contradiction excludes $k=4$, so $\ord(F)=4$ or $8$.
\end{proof}

For the rest of this section, we follow the assumptions and notation of Lemma~\ref{lem:E8}. Recall that
\[
 C=S_{H_0}(\mathrm{II}_{4,20})\simeq E_8(-2),
 \qquad R=C^\perp_{\mathrm{II}_{4,20}}.
\]

\begin{definition}\label{def:neighbor}
Consider $\gamma:A_C\to A_R$, the anti-isometry induced by unimodular gluing. The isometry $F$ induces an action on $A_C$, which we also denote by $F$, and put
\[
 W=\{F^{\operatorname{ord}(F)/4}(\alpha)-\alpha:
       \alpha\in A_C\}.
\]
Let $K$ be the overlattice of $R$ defined by
\[
 K=\{x\in R^\vee:x+R\in\gamma(W)\}.
\]
\end{definition}

\begin{proposition}\label{prop:neighbor}\label{sec:gluing}
The lattice $K$ of Definition~\ref{def:neighbor} is an $F$-invariant even unimodular overlattice of $R$:
\begin{equation}\label{eq:neighbor}
 R\subset K\subset R^\vee,\qquad [K:R]=16,\qquad K\simeq \mathrm{II}_{4,12}.
\end{equation}
Let $g=F|_K$. Then $\ord(g)=\ord(F)/2$, and on $A_R$ the involution $h=g^{\ord(F)/4}$ satisfies
\[
 \im(h-1)=\ker(h-1)=K/R.
\]
\end{proposition}
\begin{proof}
Let $d=\ord(F)/2$. We have a natural identification
\[
 C/2C\cong A_C,\qquad c+2C\mapsto c/2+C.
\]
Let $E$ be the underlying group of $C$ with form $(x,y)_E=-\frac12(x,y)_C$; then $E\simeq E_8$. Let $u=F|_E$ and $w=u^{d/2}$. Since $w^2=-1$,
\[
 (w-1)^2=-2w.
\]
Hence $2E\subset(w-1)E$, and taking determinants gives $|\det(w-1)|=2^4$. Under $C/2C=E/2E$, we have $W=(w-1)E/2E$, so
\[
 |W|=[(w-1)E:2E]
 =\frac{[E:2E]}{[E:(w-1)E]}
 =\frac{2^8}{2^4}=2^4.
\]

Next, $w^{-1}=-w$ implies $(wx,x)_E=0$ for every $x\in E$. Thus $((w-1)x)^2_E=2x^2_E$, and the discriminant quadratic form satisfies
\[
 q_C((w-1)x+2E)
 =-\frac{((w-1)x)^2_E}{2}
 =-x^2_E\equiv0\pmod{2\Z}.
\]
Here the last congruence relies on the evenness of $E$. Therefore, $W$ is totally isotropic. Since $\gamma$ is an anti-isometry, $\gamma(W)$ is also totally isotropic, and its inverse image $K\subset R^\vee$ is even. Moreover, $K/R\cong\gamma(W)$ gives $[K:R]=2^4$. Since $|A_R|=|A_C|=2^8$,
\[
 |\det K|=\frac{|\det R|}{[K:R]^2}
 =\frac{2^8}{(2^4)^2}=1.
\]
Thus $K$ is unimodular. Its signature is $(4,12)$, so $K\simeq\mathrm{II}_{4,12}$, giving \eqref{eq:neighbor}.

Since $W$ is $F$-invariant and $\gamma$ is $F$-equivariant, $K$ is $F$-invariant. On $E/2E$, the identity $(w-1)^2=0$ implies $\im(w-1)\subset\ker(w-1)$. Both have dimension $4$, so $\im(w-1)=\ker(w-1)=W$. Transporting this equality by $\gamma$ gives $\im(h-1)=\ker(h-1)=K/R$.

Since $u^d=-1\equiv1\pmod{2}$ but $u^{d/2}=w\not\equiv1\pmod{2}$, the action of $u$ on $A_C$ is of order $d$. The same holds on $A_R$, while $F^d=H_0$ fixes $R$ and hence $K$. Thus $\ord(g)=d$.
\end{proof}

\begin{lemma}\label{lem:lift}
If
\begin{equation}\tag{P}\label{eq:P}
 b\in S_g(K),\qquad b^2\ge0,\qquad 0<(b,v)\le n-1,
\end{equation}
then $(N,v)$ has a short vector.
\end{lemma}
\begin{proof}
Identify $C=E_8(-2)$, with the roots of $E_8$ having square $2$. Choose $x\in E_8$ so that the class of $x/2\in C^\vee$ maps to $b+R\in A_R$ under the gluing anti-isometry. The covering radius of $E_8$ is $1$ in this normalization \cite[Ch.~21, \S~3.D]{MR1662447}; hence some $y\in E_8$ satisfies $\|x/2-y\|^2\le1$. Thus $c=(x-2y)/2$ represents the same class and has
\[
c^2=-\frac12\|x-2y\|^2\ge-2.
\]
Then $a=b+c\in\M$ has square at least $-2$ and $(a,v)=(b,v)$. Since $C$ has no $F$-fixed vectors, $(K^g)_\Q=(\M^F)_\Q$, so $a\in N$. The primitivity of $v$ in $\M$ implies $\Q v\cap\M=\Z v$; as $0<(a,v)\le n-1<v^2$, we have $a\notin\Q v$.
\end{proof}

The following lemma is a natural continuation of Proposition~\ref{prop:neighbor}. We will use it repeatedly to rule out $\ord(F)=4$ and $8$.

\begin{lemma}\label{lem:fourspace}
Suppose an involution $h\in O(K)$ preserves $R$ and satisfies, on $A_R$,
\[
 \im(h-1)=\ker(h-1)=K/R.
\]
Let $K_-=K^{-h}$ and $K_+=K^h$. Then
\[
 \rk K_- -\ell(A_{K_-})\ge4,\qquad
 \rk K_- +\ell(A_{K_-})\le12.
\]
The map $K_-\to K/R$ is surjective with kernel $K_-\cap R$ containing $2K_{-}^{\vee}$. The kernel of the induced map $K_-/2K_-^\vee\to K/R$ is nondegenerate for the alternating pairing $(x,y)\pmod{2}$. For each of the lattices $K_+$ and $K_-$, $\rk K_\pm-\ell(A_{K_\pm})=4$ implies $K_\pm\cap R=2K_\pm^\vee$.
\end{lemma}
\begin{proof}
The space $V=2R^\vee/2K\subset K/2K$ is four-dimensional and totally isotropic. The inverse image of its annihilator with respect to the bilinear form in $K$ is $R$. For $w,w'\in2R^\vee$, set
\[
B(\overline{w},\overline{w}')=\left(\frac{(h-1)w}{2},w'\right)\pmod{2}.
\]
The isomorphism $(h-1):R^\vee/K\to K/R$ and the perfect pairing $K/R\times V\to\F_2$ induced by $(\cdot,\cdot)\pmod{2}$ show that $B$ is well-defined and nondegenerate. Writing $w=2x$, we have $(h-1)x\in K$ and obtain $B(\overline{w},\overline{w})=-((h-1)x)^2=0\pmod{2}$.

Since $(h-1)(w/2)\in K$, the vectors $(w\pm hw)/2$ belong to $K_\pm$, respectively. Replacing $w$ by $w+2k$, with $k\in K$, adds $k\pm hk$ to $(w\pm hw)/2$. Since $(k\pm hk)/2\in K_\pm^\vee$, the resulting classes modulo $2K_\pm^\vee$ are unchanged. Thus there are well-defined maps
\[
 j_\pm:V\longrightarrow K_\pm/2K_\pm^\vee,\qquad
 \overline{w}\longmapsto\frac{w\pm hw}{2}+2K_\pm^\vee.
\]
Their induced forms are $B$ and $-B$, respectively. Since $B$ is nondegenerate, both maps are injective. Unimodular gluing makes $K_\pm$ even and $2$-elementary, with equal discriminant lengths. The quotients $K_\pm/2K_\pm^\vee$ therefore have nondegenerate alternating pairings and are of dimension $\rk K_\pm-\ell(A_{K_\pm})$. The embeddings give $\rk K_\pm-\ell(A_{K_\pm})\geq 4$. Since $\ell(A_{K_+})=\ell(A_{K_-})$ and $\rk K_++\rk K_-=16$, we obtain
\[
 \rk K_-+\ell(A_{K_-})
 =16-\bigl(\rk K_+-\ell(A_{K_+})\bigr)\leq12.
\]
The same argument gives $\rk K_++\ell(A_{K_+})\leq12$.

For $z\in K_\pm$ and $w\in2R^\vee$, we have $\left(z,\frac{w\pm hw}{2}\right)=(z,w)$. The nondegenerate pairing between $K/R$ and $V$ identifies $K/R$ with $\operatorname{Hom}(V,\F_2)$.
Under this identification, the map $K_\pm\to K/R$ sends $x$ to $\xi\mapsto(x,j_\pm(\xi))\bmod2$. It factors through $K_\pm/2K_\pm^\vee$ and is surjective. Its kernel is the orthogonal complement of the nondegenerate subspace $j_\pm(V)$, hence is nondegenerate. When the quotient has dimension $4$, this kernel is zero, giving $K_\pm\cap R=2K_\pm^\vee$.
\end{proof}

\Needspace{12\baselineskip}
\subsection{Nonnegative Short Vectors}
The preceding rank bound leads to the following nonnegative short-vector theorem.

\begin{theorem}\label{thm:nonnegative}
Let $L$ be an even $2$-elementary lattice of signature $(1,r-1)$, with $r+\ell(A_L)\le12$. For every $v\in L$ with $v^2=2(n-1)\ge4$, there exists $b\in L$ with $b^2\ge0$ and $0<(v,b)\le n-1$.
\end{theorem}
\begin{proof}
If $L_0=L\oplus D$ with $D$ negative definite and $b\in L_0$ satisfies the conclusion for $v\in L$, then its projection $b_L\in L$ satisfies $(v,b_L)=(v,b)$ and $b_L^2\ge b^2$.

Let $a=\ell(A_L)$. Since the Gram form of an even lattice is alternating modulo $2$, $r-a$ is even. Thus
\[
 k=\frac{12-r-a}{2}\in\Z_{\ge0},\qquad
 L^\sharp=L\oplus A_1(-1)^k,
 \qquad \rk L^\sharp+\ell(A_{L^\sharp})=12.
\]
By the projection argument, we may therefore assume $r+a=12$. For $x\in L^\vee$ and $\overline{x}=x+L$, the discriminant quadratic form is $q_L(\overline{x})=x^2\pmod{2\Z}$. Put $\delta=0$ if $q_L$ is integer-valued and $\delta=1$ otherwise. Since $a\le r$, we have $6\le r\le12$; the case $r=12$, $a=0$ is excluded by the signature condition for even unimodular lattices. If $\delta=0$, the existence conditions in \cite[Theorem~3.6.2]{MR525944} leave only $(r,a)=(10,2)$. Thus the block decomposition \cite[\S~9.2]{MR4844629} and indefinite uniqueness \cite[Theorem~3.6.2]{MR525944} give exactly $7$ triples:
\[
 (6,6,1),\ (7,5,1),\ (8,4,1),\ (9,3,1),\
 (10,2,0),\ (10,2,1),\ (11,1,1).
\]
They are represented, in order, by $I_{1,5}(2)$, $U\oplus A_1(-1)^5$, and $U\oplus Q(-1)$ for
\[
 Q=D_4\oplus A_1^2,\ D_6\oplus A_1,\ D_8,\ E_7\oplus A_1,\ E_8\oplus A_1.
\]
Lemmas~\ref{lem:rootmodels} and \ref{lem:exceptional} below prove the assertion for these models.
\end{proof}

We verify the models using Voronoi cells and two explicit reductions.

\begin{definition}\label{def:voronoi}
For a positive-definite lattice $Q$, the Voronoi cell centered at $0$ is defined by the nearest-point condition \cite[Ch.~2, \S~1.2]{MR1662447}:
\[
 \operatorname{Vor}(Q)=\{x\in Q\otimes_{\Z}\R:\|x\|\le\|x-q\|\text{ for all }q\in Q\}.
\]
\end{definition}

\begin{lemma}\label{lem:voronoi}
Let $Q$ be a positive-definite even lattice. Suppose $\operatorname{Vor}(Q)$ is contained in the convex hull of finitely many vectors $h$, each of which either has $h^2\le3/2$ or has $h^2=2$ with $h\in Q^\vee$ and $2h\in Q$.
For every $v\in U\oplus Q(-1)$ with $v^2=2(n-1)\ge4$, there exists $b\in U\oplus Q(-1)$ with $b^2\ge0$ and $0<(v,b)\le n-1$.
\end{lemma}
\begin{proof}
Use $\|z\|^2$ and $(z,h)_Q$ for the norm and pairing in the positive-definite lattice $Q$. Choose a hyperbolic basis $e,f$ of $U$ so that $v=\alpha e+\beta f+z$ with $\alpha,\beta>0$ and
\[
 n-1=\alpha\beta-\frac12\|z\|^2.
\]
For $t\in Q$, the Eichler translation fixing $f$ is the isometry of $U\oplus Q(-1)$ given by
\[
 \alpha e+\beta f+z
 \longmapsto
 \alpha e+
 \left(\beta+(z,t)_Q+\frac{\alpha}{2}\|t\|^2\right)f
 +z+\alpha t,
\]
where $(\, ,\,)_Q$ is the pairing on $Q$. Choosing $t$ appropriately, we may assume $z/\alpha\in\operatorname{Vor}(Q)$. Exchanging $e,f$ whenever $\beta<\alpha$ strictly decreases the positive integer $\alpha$. Thus we may assume $\beta\ge\alpha$ and $z/\alpha\in\operatorname{Vor}(Q)$.
If $\alpha\le n-1$, take $b=f$. Otherwise $\alpha>n-1\ge2$. The norm bound $\|z\|^2\le2\alpha^2$ gives $n-1\ge\alpha(\beta-\alpha)$, hence $\beta=\alpha$.

Because $z/\alpha$ lies in this convex hull, one of the specified vectors $h$ satisfies $(z,h)_Q\ge\|z\|^2/\alpha$. If $h^2\le3/2$, Cauchy's inequality gives $\|z\|^2\le3\alpha^2/2$ and $n-1\ge\alpha^2/4$. For integers $\alpha\ge3$, this contradicts $n-1<\alpha$, using integrality when $\alpha=3$.

Otherwise $h^2=2$. Since $(z,h)_Q\le\sqrt2\|z\|<2\alpha$, the integer $\delta=2\alpha-(z,h)_Q$ is positive. The support inequality gives $n-1\ge\alpha\delta/2$, so $\delta=1$. The integral vector $b=2e+2f+2h\in U\oplus Q(-1)$ is isotropic and has $(v,b)=2\delta=2\le n-1$.
\end{proof}

\begin{lemma}\label{lem:rootmodels}
Let $Q$ be any of the lattices
\[
 D_4\oplus A_1^2,\quad D_6\oplus A_1,\quad D_8,\quad
 E_7\oplus A_1,\quad E_8\oplus A_1.
\]
For every $v\in U\oplus Q(-1)$ with $v^2=2(n-1)\ge4$, there exists $b\in U\oplus Q(-1)$ with $b^2\ge0$ and $0<(v,b)\le n-1$.
\end{lemma}
\begin{proof}
For an irreducible root lattice of rank $j$, the part of its Voronoi cell in a dominant Weyl chamber satisfies $(x,\theta)\le1$, where $\theta=\sum_{i=1}^j c_i\alpha_i$ is the highest root \cite[Ch.~21, Theorem~5]{MR1662447}. It is therefore contained in the simplex with vertices $0,\omega_i/c_i$, where $\omega_i$ are the fundamental weights. The images of these vertices under the Weyl group, and their
Cartesian products for direct sums, give the required finite set.

The nonzero $A_1$ vertex has square $1/2$. For $D_j$, the spinor vertices $\frac12(1,\ldots,1,\pm 1)$ have square $j/4$; all other vertices have square at most $\max(1,(j-2)/4)$. In the numbering with the $E_7$ chain $1,\ldots,6$ and vertex $7$ attached to $3$, the highest-root coefficients $(2,3,4,3,2,1,2)$ give one vertex $\omega_6$ of square $3/2$; all the others have square at most $1$.

The squared covering radius of $E_8$ is one \cite[Ch.~21, \S~3.D]{MR1662447}, so every Voronoi vertex of $E_8\oplus A_1$ has square at most $3/2$. For the other $4$ lattices, every product vertex has square at most $3/2$, except for vertices of square $2$ formed from a square-one $D_4$ vertex and two $A_1$ vertices, a $D_6$ spinor vertex and one $A_1$ vertex, a $D_8$ spinor vertex, or $\omega_6$ and one $A_1$ vertex. All these vertices lie in $Q^\vee$, and their doubles lie in $Q$: this holds for the $A_1$ vertices, the relevant $D_4$ vertices, the spinor weights of even $D_j$, and $\omega_6$ since $|A_{E_7}|=2$. Lemma~\ref{lem:voronoi} applies.
\end{proof}

\begin{lemma}\label{lem:exceptional}
Let $L$ be one of the lattices
\[
 U\oplus A_1(-1)^5,\qquad I_{1,5}(2).
\]
For every $v\in L$ with $v^2=2(n-1)\ge4$, there exists $b\in L$ with $b^2\ge0$ and $0<(v,b)\le n-1$.
\end{lemma}
\begin{proof}
For $U\oplus A_1(-1)^5$, write
\[
 v=\alpha e+\beta f+\sum_{i=1}^5 c_i r_i,\qquad r_i^2=-2.
\]
Eichler transformations, sign changes, and exchanges give $0<\alpha\le\beta$ and $0\le c_i\le\alpha/2$. Also arrange
\[
 p=\alpha+\beta-\sum_{i=1}^5 c_i\ge0
\]
by reflecting in $2e+2f+\sum_{i=1}^5 r_i$, a root. If $p<0$, this replaces $\alpha,\beta$ by $\alpha+4p,\beta+4p$ and decreases their positive sum; Reducing each $c_i$ modulo $\alpha$ to $[-\alpha/2,\alpha/2]$ by Eichler translations keeps $\alpha$ fixed and does not increase $\beta$. Hence the process terminates.

If $\alpha\le n-1$, use $f$. Otherwise, let $t=\min_i c_i$. Since $t\le c_i\le\alpha/2$,
\[
 c_i^2\le\left(\frac{\alpha}{2}+t\right)c_i-\frac{\alpha t}{2}.
\]
Summing and using $\sum_{i=1}^5 c_i=\alpha+\beta-p$ gives
\[
 n-1\ge\frac{\alpha}{2}(p+t)
 +(\beta-\alpha)\left(\frac{\alpha}{2}-t\right)+pt
 \ge\frac{\alpha}{2}(p+t).
\]
As $n-1<\alpha$, the nonnegative integer $p+t$ is at most $1$. If $c_{i_0}=t$, then $b=2e+2f+\sum_{\substack{i=1\\i\ne i_0}}^5 r_i$ is a nonzero isotropic vector with $(v,b)=2(p+t)$. The orthogonal complement of the positive vector $v$ is negative definite, so $p+t=1$ and $(v,b)=2\le n-1$.

Finally take $I_{1,5}(2)$ with integral coordinates $v=(a;b_1,\ldots,b_5)$, $a>0$, and $b_1\ge\cdots\ge b_5\ge0$. Reflect in $(1;1,1,1,0,0)$ to arrange $a\ge b_1+b_2+b_3$. Let $p=a-b_1>0$, so $p\ge b_2+b_3$. Then
\[
 n-1\ge p^2+2pb_2-b_2^2-3b_3^2\ge p^2,
\]
because $2pb_2-b_2^2-3b_3^2\ge(b_2-b_3)(b_2+3b_3)\ge0$. The isotropic vector $u=(1;1,0,\ldots,0)$ satisfies $(v,u)=2p\le n-1$: for $p=1$ this follows from $n-1\ge2$, while for $p\ge2$ we have $2p\le p^2\le n-1$.
\end{proof}

\subsection{Order Four And Eight Cases}\label{sec:orderfour}\label{sec:remaining}

\begin{proposition}\label{prop:orderfour}
If $(N,v)$ has no short vector, then $\ord(F)\ne4$.
\end{proposition}
\begin{proof}
If $\ord(F)=4$, then $g$ is an involution and $S_g(K)=K^{-g}$ is even, hyperbolic, and $2$-elementary. Lemma~\ref{lem:fourspace} gives $\rk S_g(K)+\ell(A_{S_g(K)})\le12$. Theorem~\ref{thm:nonnegative} supplies \eqref{eq:P}, a contradiction to Lemma~\ref{lem:lift}.
\end{proof}

It remains to exclude $\ord(F)=8$.

\begin{lemma}\label{lem:shortlist}
In the setting of Proposition~\ref{prop:neighbor}, suppose $\ord(F)=8$, so $\ord(g)=4$. Let $C'=K^{-g^2}$, and write $q=\rk C'$ and $a=\ell(A_{C'})$.
Then $C'\subset S_g(K)\cap v^\perp$ is negative definite,
\[
 q-a\ge4,\qquad q+a\le12,
\]
and the quotient map $C'\to K/R$ is surjective and annihilates $2(C')^\vee$. Moreover,
\[
 C'(-1)\simeq D_6,\ D_{10},\ \text{or }D_4^2.
\]
\end{lemma}
\begin{proof}
The assertions about signature and orthogonality follow from $g^2(v)=v$ and the fixed positive three-space. The involution $g^2$ is the one in Proposition~\ref{prop:neighbor}. Lemma~\ref{lem:fourspace} gives the inequalities and the quotient map. Any root in $C'\cap R\subset N\cap v^\perp$ would satisfy \eqref{eq:W}, so $C'\cap R$ is rootless. Since the quotient map annihilates $2(C')^\vee$, a root of divisibility two would belong to $C'\cap R$. Thus $C'$ has no such roots.

Since $g^2=-1$ on $C'$, its rank $q$ is even.
As $C'$ is even and $2$-elementary, $q-a$ is even, so $a$ is also even. Combining the two inequalities with the existence conditions for even $2$-elementary forms \cite{MR525944} and \cite[Theorem 3.6.3]{MR3378829} gives
\[
 (q,a)=(6,2),(8,0),(8,2),(8,4),(10,2).
\]

Set $E=C'(-1)$ and identify their underlying groups. For $(q,a)=(6,2)$, the finite quadratic form is $\langle-1/2\rangle^2$. Gluing with $A_1^2$ gives $E_8$. Orthogonal pairs of roots in $E_8$ are conjugate and have complement $D_6$, so $E=D_6$.

Since $E$ is positive definite of rank $8$, its discriminant form has Brown invariant (signature modulo $8$) equal to $0$ \cite[Theorem~1.3.3]{MR525944}. Finite quadratic-form classification \cite[Theorem~3.6.3]{MR3378829} gives a totally isotropic subgroup of order $2^{a/2}$; gluing along it produces an even unimodular overlattice $\Lambda\simeq E_8$ of index $2^{a/2}$. Since $E$ is $2$-elementary, $2\Lambda\subset E$. Let $V=\Lambda/2\Lambda$ with $Q(x)=x^2/2\pmod{2}$, and let $Z=(E/2\Lambda)^\perp$.
Then $\dim Z=a/2$, and
\[
 Z=2E^\vee/2\Lambda\subset Z^\perp=E/2\Lambda,
 \qquad E/2E^\vee\simeq Z^\perp/Z.
\]
By Lemma~\ref{lem:fourspace},
\[
 2\Lambda\subset2E^\vee\subset E\cap R.
\]
Every class of $V$ with $Q=1$ has a root representative: the squared covering radius of $E_8$ is $1$, so each class has a representative of square at most $4$. Since $E\cap R$ is rootless and contains $2\Lambda$, $Q$ vanishes on $(E\cap R)/2\Lambda$. In particular, it vanishes on $Z$ and descends to $E/2E^\vee$. Its restriction to
\[
 H=(E\cap R)/2E^\vee=\ker(E/2E^\vee\to K/R)
\]
is therefore zero. Lemma~\ref{lem:fourspace} makes the polarization on $H$ nondegenerate, so $H=0$. Since $\dim H=4-a$, we obtain $a=4$ and $Z$ is a totally singular plane. Root reflections realize $O(V,Q)$ \cite{MR1662447}; by Witt's lemma, $Z$ is equivalent to the plane spanned by $2e_1$ and $e_1+e_2+e_3+e_4$ in standard coordinates. The orthogonal preimage $E\subset\Lambda$ is $D_4\oplus D_4$: orthogonality to $2e_1$ removes the half-integral vectors, while orthogonality to $e_1+e_2+e_3+e_4$ imposes even sums in each four-coordinate block.

For $(q,a)=(10,2)$, the discriminant form is $\langle1/2\rangle^2$. Gluing with $D_6$ produces an even unimodular rank-sixteen lattice, either $E_8^2$ or $D_{16}^+$ \cite{MR1662447}. In the former, the connected $D_6$ root system lies in one $E_8$ factor and has complement $E_8\oplus A_1^2$. In the latter case, after permuting coordinates, the roots of this $D_6$ subsystem are $\pm e_i\pm e_j$ with $1\le i<j\le6$, where $e_1,\ldots,e_{16}$ is the standard orthonormal basis. Orthogonality forces these coordinates to vanish, which excludes every vector of $D_{16}^+\setminus D_{16}$, whose coordinates all lie in $\Z+\tfrac12$. The remaining $10$ coordinates
are integers with even sum, so the complement is $D_{10}$.
\end{proof}

\begin{proposition}\label{prop:compatible}
Let $K$ be even and unimodular of signature $(4,12)$. Let $g\in O(K)$ have order $4$ and fix a positive-definite three-space pointwise. If $K^{-g^2}\simeq D_6(-1)$, then every $v\in K$ with $g(v)=-v$ and $v^2=2(n-1)\ge4$ admits $b\in S_g(K)$ with $b^2\ge0$ and $0<(v,b)\le n-1$.
\end{proposition}
\begin{proof}
Let $C'=K^{-g^2}$ and $R_1=(C')^\perp_K$. Every isometry of $D_6$ is a signed permutation. Since $(g|_{C'})^2=-1$, its permutation consists of $3$ two-cycles, each with negative sign product. Thus $g$ exchanges the two spinor classes of $A_{C'}$ and hence those of $A_{R_1}$.

Let $\tau$ be the sign change of the first coordinate in $D_6(-1)$. It exchanges the spinor classes $\tfrac12(1,1,1,1,1,\pm 1)+C'$ and fixes their sum, so $\tau$ and $g|_{C'}$ induce the same action on $A_{C'}$. Since $g^2=1$ on $R_1$, unimodular gluing extends $h=(g|_{R_1})\oplus\tau$ to an involution of $K$. Let $L=K^{-h}$, $r=\rk L$, and $a=\ell(A_L)=\ell(A_{K^h})$. Unimodular gluing makes both eigensublattices even and $2$-elementary.

The fixed positive three-space lies in $K^h$. Moreover $v\in L$, while $K^g\subset K^h$ because $g$ has no fixed vectors in $C'$. Hence $L$ is hyperbolic and $L\subset S_g(K)$.

Inside $D_6(-1)$, the coordinate vectors supported on positions $2,\ldots,6$ and summing to zero form $A_4(-1)\subset K^h$. Its determinant is $5$, so the Gram matrix of $K^h$ has rank at least $4$ modulo $2$. Since $K^h$ is $2$-elementary, this rank equals $\rk(K^h)-a=16-r-a$. Thus $r+a\le12$. Theorem~\ref{thm:nonnegative} applied to $v\in L$ gives $b\in L\subset S_g(K)$ satisfying \eqref{eq:P}.
\end{proof}

\begin{proposition}\label{prop:remaining}
If $(N,v)$ has no short vector, then $\ord(F)\ne8$.
\end{proposition}
\begin{proof}
Suppose $\ord(F)=8$, and use the notation of Lemma~\ref{lem:shortlist}. If $C'=D_6(-1)$, Proposition~\ref{prop:compatible} gives \eqref{eq:P}, a contradiction to Lemma~\ref{lem:lift}.

The remaining possibilities are $C'=D_{10}(-1)$ and $D_4(-1)^2$, both with $q+a=12$. Lemma~\ref{lem:fourspace} gives $\rk(K^{g^2})-\ell(A_{K^{g^2}})=16-q-a=4$ and therefore $K^{g^2}\cap R=2(K^{g^2})^\vee$.
Since $v\in K^{g^2}\cap R$, we have $v/2\in(K^{g^2})^\vee$. Every discriminant class of $D_4(-1)^2$ has a representative of square at least $-2$; for $D_{10}(-1)$ the corresponding bound is $-5/2$, attained by the spinor classes. Choose a representative $c\in(C')^\vee$ corresponding to the class of $v/2$. Then $b=v/2+c\in K$ and $(b,v)=n-1$. Its square satisfies
\[
 b^2\ge\frac{n-1}{2}-\frac52\ge-\frac32>-2.
\]
Evenness of $K$ gives $b^2\ge0$. Both summands are perpendicular to $K^g$, so \eqref{eq:P} holds, a contradiction to Lemma~\ref{lem:lift}.
\end{proof}

\section{Proof of The Main Theorem}\label{sec:conclusion}
Let $X$ be a hyperk\"ahler manifold of $K3^{[n]}$-type, $n\ge2$, and choose a representative of its canonical Mukai embedding orbit, so that $H=H^2(X,\Z)=v^\perp\subset\M$ and $v^2=2(n-1)$. Write $\mathcal K_X$ for its K\"ahler cone.

\Needspace{7\baselineskip}
\begin{definition}
\label{def:numerical-wall}
For a primitive class $D\in\operatorname{Pic}(X)$ with $D^2<0$, put
\[
T_D=\M\cap\operatorname{span}_{\Q}\{v,D\}.
\]
The class $D$ is a numerical wall divisor if there exists $a\in T_D\setminus\Q v$ such that
\[
a^2\ge-2,\qquad 0\le(a,v)\le v^2/2.
\]
Equivalently, there exists $a\in T_D\setminus\Q v$ satisfying either
\[
a^2=-2,\quad 0\le(a,v)\le v^2/2,
\qquad\text{or}\qquad
0\le a^2v^2<(a,v)^2\le(v^2/2)^2.
\]
\end{definition}

The second condition is from \cite[Definition~9]{MR3473636}. Notice that \cite[Definition~9]{MR3473636} contains a typo: the last two inequality signs, $\le$ and $<$, should be interchanged.

The series work about wall classes in \cite{BayerMacri}, \cite[Theorem~1]{MR3377069}, \cite[Theorems~1.3 and 2.4]{MR3423735}, and \cite[Theorem~10]{MR3473636} tells that every numerical wall divisor $D$ satisfies $D^\perp\cap\mathcal K_X=\varnothing$.

\begin{proof}[Proof of Theorem~\ref{thm:main}]
Average a K\"ahler class to obtain a $G$-invariant class $\kappa\in\mathcal K_X$. The action on $A_H$ is automatically trivial for $n=1,2$. For $n\ge3$, suppose that some $f\in G$ acts nontrivially on $A_H$. Markman's monodromy description \cite[Lemma~9.2]{MR2964480} gives $f^*|_{A_H}=-\id$. Unimodular gluing in the canonical embedding \cite[Corollary~9.5(1)]{MR2964480} extends $f^*$ and $v\mapsto-v$ to a finite-order $F\in O(\M)$. Since $f$ is symplectic, $F$ fixes the positive three-space spanned by $\kappa,\mathrm{Re}\,\sigma,\mathrm{Im}\,\sigma$. Theorem~\ref{thm:lattice} therefore gives a short vector $a\in S_F(\M)$ for $v$; namely, after choosing its sign,
\[
a\notin\Q v,\qquad a^2\ge-2,\qquad 0\le(a,v)\le v^2/2.
\]

The lattice $S_f(H)$ is negative definite and algebraic, and $S_F(\M)_\Q=\Q v\oplus S_f(H)_\Q$. Thus the projection
\[
c=a-\frac{(a,v)}{v^2}v\in S_f(H)_\Q\setminus\{0\}
\]
is algebraic and has negative square. Choose a primitive class $D\in S_f(H)$ spanning $\Q c$. Moreover, $a\in T_D\setminus\Q v$. Definition~\ref{def:numerical-wall} shows that $D$ is a numerical wall divisor. But $\kappa$ is invariant, so $(D,\kappa)=0$. Hence
\[
\kappa\in D^\perp\cap\mathcal K_X,
\]
a contradiction to the wall-divisor property. Thus $G$ acts trivially on $A_H$.

For the second assertion, when $n\ge2$, extend the $G$-action to $\M$ fixing $v$, using its trivial action on $A_H$; the coinvariant lattice is unchanged. Unimodular gluing then gives triviality on $A_{S_G(H)}$, and for $n=1$ the same argument applies directly to the unimodular lattice $H$.
\end{proof}

\providecommand{\bysame}{\leavevmode\hbox to3em{\hrulefill}\thinspace}
\providecommand{\MR}{\relax\ifhmode\unskip\space\fi MR }
\providecommand{\MRhref}[2]{%
  \href{http://www.ams.org/mathscinet-getitem?mr=#1}{#2}
}

\Addresses
\end{document}